\documentclass[a4paper,10pt]{article}
\usepackage{mathtext}
\usepackage[T1,T2A]{fontenc}
\usepackage[cp1251]{inputenc}
\usepackage[english]{babel}
\usepackage{amsmath}
\usepackage{amsfonts}
\usepackage{amssymb}
\usepackage{mathrsfs}
\usepackage{amsthm}
\usepackage{titling}
\usepackage{multicol}
\usepackage{graphicx}
\usepackage{wrapfig}
\usepackage{enumerate}

\usepackage{color}
\usepackage{euscript}

\DeclareMathOperator{\sign}{sign}

\newcommand{\Bell}{\boldsymbol{B}}

\newcommand{\BMO}{\mathrm{BMO}}
\newcommand{\eps}{\varepsilon}

\DeclareMathOperator{\E}{\mathbb{E}}
\newcommand{\av}[2]{\langle {#1}\rangle_{{}_{#2}}}
\renewcommand{\le}{\leqslant}

\renewcommand{\leq}{\leqslant}
\renewcommand{\geq}{\geqslant}
\renewcommand{\emptyset}{\varnothing}

\renewcommand{\phi}{\varphi}
\newcommand{\Set}[2]{\Big\{{#1}\colon{#2}\Big\}}

\newcommand{\eqlb}[2]{\begin{equation}\label{#1} #2 \end{equation}}

\newcommand{\vf}{\varphi}

\newtheorem{Le}{Lemma}[section]

\newtheorem{St}[Le]{Proposition}
\newtheorem{Th}[Le]{Theorem}
\newtheorem{Cor}[Le]{Corollary}
\newtheorem{Rem}[Le]{Remark}

\numberwithin{equation}{section}

\begin{document}
\author{Dmitriy~Stolyarov \and Vasily Vasyunin \and Pavel Zatitskiy}
\title{Sharp mutliplicative inequalities with $\BMO$~$\mathrm{I}$\thanks{Support by the Russian Science Foundation grant 19-71-10023.}}

\maketitle
\begin{abstract}
We find the best possible constant~$C$ in the inequality $\|\varphi\|_{L^r}\leq C\|\varphi\|_{L^p}^{\frac{p}{r}}\|\varphi\|_{\BMO}^{1-\frac{p}{r}}$, where~$2 \leq r$ and~$p < r$. We employ the Bellman function technique to solve this problem in the case of an interval and then transfer our results to the circle and the line. 

\medskip


\end{abstract}

\section{Introduction}
The space~$\BMO$ plays an important role in analysis. It serves as a good substitute for~$L^{\infty}$ in the endpoint estimates that fail for the latter space. On the other hand, some estimates that are trivial for~$L^{\infty}$ become more interesting for the case of~$\BMO$. Consider the classical multiplicative inequality
\begin{equation}\label{MultiplicativeLebesgue}
\|\varphi\|_{L^r}\leq \|\varphi\|_{L^p}^{\frac{p}{r}}\|\varphi\|_{L^{\infty}}^{1-\frac{p}{r}}, \quad 0 < p < r < \infty.
\end{equation}
This inequality holds true for any function~$\varphi$ on a measurable space, and follows from the simple estimate~$|\vf(x)| \leq \|\vf\|_{L^{\infty}}$. The bound~\eqref{MultiplicativeLebesgue} is sharp. 

The inequality
\begin{equation}\label{FirstMultiplicative}
\|\varphi\|_{L^r}\leq C\|\varphi\|_{L^p}^{\frac{p}{r}}\|\varphi\|_{\BMO}^{1-\frac{p}{r}},\qquad 1\leq p < r < \infty,
\end{equation}
is less trivial. Here~$C$ is a constant that should depend neither on~$\vf$, nor on~$p$, nor on~$r$. For this inequality, one needs to specify the domain of~$\vf$. In the case of the Euclidean space~$\mathbb{R}^d$, the inequality
\begin{equation*}
\|\varphi\|_{L^r(\mathbb{R}^d)}\leq C_d\|\varphi\|_{L^p(\mathbb{R}^d)}^{\frac{p}{r}}\|\varphi\|_{\BMO(\mathbb{R}^d)}^{1-\frac{p}{r}}, \qquad 1\leq p < r < \infty,
\end{equation*}
was first obtained in~\cite{ChenZhu}. It plays an important role in interpolation and extrapolation theory for~$\BMO$, see~\cite{Milman}. We note that the exponents~$\frac{p}{r}$ and~$1-\frac{p}{r}$ in the inequalities above are also dictated by interpolation theory, see~\cite{Milman}. Alternatively, one may restore the exponents from the dilation invariance of~\eqref{MultiplicativeLebesgue} and~\eqref{FirstMultiplicative}. Our aim is to obtain sharp versions of~\eqref{FirstMultiplicative}. For that we need to specify the choice of the~$\BMO$ norm. 

The progress in computation of sharp constants in the John--Nirenberg type inequalities for~$\BMO$ in higher dimensions is scant. In fact, even the asymptotical behavior of these constants is unknown, see~\cite{CSS}. So, we limit ourselves to the case~$d=1$. Let us consider the case of the~$\BMO$ space on an interval~$I$. A real-valued function~$\varphi\in L^1(I)$ belongs to~$\BMO(I)$ provided the quantity
\begin{equation}\label{BMOnorm}
\|\varphi\|_{\BMO(I)}^2 =\!\!\!\!\! \sup\limits_{\genfrac{}{}{0pt}{3}{J\,\hbox{\tiny is a}}{\hbox{\tiny subinterval of}\, I}}\!\!\!\av{\,|\varphi - \av{\varphi}{J}|^2}{J}
\end{equation}
is finite. Here and in what follows we use the notation~$\av{\psi}{E}$ to denote the average of a function~$\psi$ over a set~$E$ of positive measure, that is
\begin{equation*}
\av{\psi}{E} = \frac{1}{|E|}\int\limits_E\psi.
\end{equation*}
The choice of the exponent~$2$ in the definition~\eqref{BMOnorm} does not affect the validity of~\eqref{FirstMultiplicative} since, by the John--Nirenberg inequality, one obtains an equivalent norm for other values of~$p$. However, this choice is important for sharp constants. Most of the work related to sharp constants for~$\BMO$ functions was done with the quadratic norm. However, see~\cite{Korenovskii},~\cite{Lerner},~\cite{Slavin}, and~\cite{SlavinVasyunin} for the results concerning the classical~$1$-norm and arbitrary~$p$-norm.

We warn the reader that the inequality~\eqref{FirstMultiplicative} cannot be true for functions on the interval since the seminorm~\eqref{BMOnorm} vanishes on constant functions. So, this inequality needs a slight modification, which is our first main result.
\begin{Th}\label{MultiplicativeTheoremInterval}
The inequality
\eqlb{eq141202}{%\begin{equation*}
\|\varphi\|_{L^r(I)} \leq \Big(\frac{\Gamma(r+1)}{\Gamma(p+1)}\Big)^\frac{1}{r}\|\varphi\|_{L^p(I)}^{\frac{p}{r}}\|\varphi\|_{\BMO(I)}^{1-\frac{p}{r}},\qquad \av{\varphi}{I} = 0,
}
%\end{equation*}
holds true and is sharp when~$p\geq 1$ and $\max(2, p) \leq r < \infty$. %{\color{red} OTHER CASES???}
\end{Th}
%By the~$L^q(I)\slash\mathbb{R}$ we mean~$L^q$ factorized by constant functions, that is
%\begin{equation*}
%\|\varphi\|_{L^q(I)} = \inf\limits_{c\in\mathbb{R}}\|\varphi-c\|_{L^q(I)}.
%\end{equation*}
We must say a couple of words about our tools. We will be using the Bellman function method. It allows to derive sharp inequalities for non-compact infinite dimensional objects (such as the unit ball of the~$\BMO$ space) from certain finite dimensional boundary value problems. The papers~\cite{NT} and~\cite{NTV} laid the foundation of the method. We refer the reader to~\cite{SlavinVasyuninLecNot} and~\cite{Volberg} for the basics of the theory and to~\cite{Osekowski} for the probabilistic point of view (in the probability theory, this technique is usually called the Burkholder method).

The Bellman functions are convenient for~$\BMO$ problems. Their successful application lead to the computation of sharp constants in various forms of the John--Nirenberg inequalities and related problems, see~\cite{SV},~\cite{SV2}, and~\cite{VasyuninVolberg}. Later the branch of the Bellman function method that works with~$\BMO$ problems was converted into a theory in~\cite{Memoirs} and~\cite{SZ}.

The Bellman function appearing in our proof of Theorem~\ref{MultiplicativeTheoremInterval} is interesting in itself. All the Bellman functions in the papers cited in the previous paragraph are two-dimensional, whereas our function is three-dimensional. The two-dimensional optimization problems related to~$\BMO$ are well-understood (see~\cite{Memoirs} and~\cite{CR2}), which is not quite true for higher dimensional ones. The difficulty increases dramatically. Luckily, the Bellman function appearing in the proof of Theorem~\ref{MultiplicativeTheoremInterval} is tractable.% {\color{red} That is why we limit ourselves to the case~$\cre{\max(2, p) \leq r$. The Bellman functions for other cases are much more complicated, their description requires a separate paper. }

Using the technique developed in~\cite{SZTransference}, we will transfer our results to the circle and the line.
\begin{Th}\label{Ctheorem}
The inequality
\eqlb{MultCircle}{
\|\varphi\|_{L^r(\mathbb{T})} \leq \Big(\frac{\Gamma(r+1)}{\Gamma(p+1)}\Big)^\frac{1}{r}\|\varphi\|_{L^p(\mathbb{T})}^{\frac{p}{r}}\|\varphi\|_{\BMO(\mathbb{T})}^{1-\frac{p}{r}},\qquad \varphi \in \BMO(\mathbb{T}),\ \int\limits_{\mathbb{T}}\varphi = 0,
}%\end{equation*}
holds true and is sharp when~$p\geq 1$ and $\max(2, p) \leq r < \infty$. %{\color{red}OTHER CASES}
\end{Th}
\begin{Th}\label{Rtheorem}
The inequality
\eqlb{eq141201}{
\|\varphi\|_{L^r(\mathbb{R})} \leq \Big(\frac{\Gamma(r+1)}{\Gamma(p+1)}\Big)^\frac{1}{r}\|\varphi\|_{L^p(\mathbb{R})}^{\frac{p}{r}}\|\varphi\|_{\BMO(\mathbb{R})}^{1-\frac{p}{r}},\qquad \varphi \in L^p(\mathbb{R}),
}%\end{equation*}
holds true and is sharp when~$p\geq 1$ and $\max(2, p) \leq r < \infty$. %{\color{red}OTHER CASES}
\end{Th}
 
\begin{Rem}
If $r<p$\textup{,} then~\eqref{FirstMultiplicative} is trivially invalid with any finite $C$ in any of the cases considered. The case $1\leq p <r<2$ is substantial\textup{,} but much more complicated. We will investigate it in a separate paper.
\end{Rem}

We postulate the Bellman function problem in the forthcoming section, study its simple properties, and relate it to the already known Bellman functions. We compute the Bellman function in Section~\ref{SectG} and provide a small portion of additional information about the corresponding lower Bellman function in Section~\ref{SecLowBel}. Section~\ref{SETC} contains the proof of Theorem~\ref{MultiplicativeTheoremInterval}. Section~\ref{s6} provides the derivation of  Theorems~\ref{Ctheorem} and~\ref{Rtheorem} from Theorem~\ref{MultiplicativeTheoremInterval}.

\section{Optimization problem}\label{s2}
We introduce the main character. This is the Bellman function~$\Bell_{p,r;\eps}\colon\mathbb{R}^3\to \mathbb{R}\cup\{\pm\infty\}$. It is defined as
\begin{equation}\label{NewBellman}
\Bell_{p,r;\eps} (x_1,x_2,x_3) = \sup\Set{\av{|\phi|^r}{I}}{\|\phi\|_{\BMO(I)} \leq \eps,\; \av{\phi}{I}=x_1,\; \av{\phi^2}{I}=x_2,\; \av{|\phi|^p}{I}=x_3}.
\end{equation}

We say that $\vf$ is a \emph{test function} for the point $x\in\mathbb{R}^3$ if 
\begin{equation*}
\|\phi\|_{\BMO(I)} \leq \eps,\; \av{\phi}{I}=x_1,\; \av{\phi^2}{I}=x_2,\; \av{|\phi|^p}{I}=x_3.
\end{equation*}
The main purpose of this paper is to find an explicit formula for~$\Bell_{p,r;\eps}$. We state this result by referring to the formulas appearing in the forthcoming sections.
\begin{Th}\label{BellmanTheorem}
For $(r-2)(p-r)<0$ the function $\Bell_{p,r;\eps}$ coincides with the function $G$ given by formulas~\eqref{GDef1}\textup{,}~\eqref{Symmetry}\textup{,} and~\eqref{GDef2} on the domains~\eqref{defXi}. The functions~$m_r$ and~$k_r$ are defined by~\eqref{mp+} and~\eqref{mp-} respectively.
\end{Th}
For the case $(r-2)(p-r)>0$ the function $G$ constructed by formulas~\eqref{GDef1} and~\eqref{GDef2} will coincide with the minimal Bellman function, see Section~\ref{SecLowBel}. %Investigation of the function $\Bell_{p,r;\eps}$ for this case is much more complicated,  it will be done in another paper.

To describe the function~$\Bell_{p,r;\eps}$, we will need two auxiliary Bellman functions~$\Bell^\pm_{p;\eps}\colon\mathbb{R}^2\to\mathbb{R}\cup\{\pm\infty\}$ defined by the rule
\begin{align}\label{OldBellman}
\Bell^+_{p;\eps} (x_1,x_2) = \sup\Set{\av{|\phi|^p}{I}}{\|\phi\|_{\BMO(I)} \leq \eps,\; \av{\phi}{I}=x_1,\; \av{\phi^2}{I}=x_2},\\
\Bell^-_{p;\eps} (x_1,x_2) = \inf\Set{\av{|\phi|^p}{I}}{\|\phi\|_{\BMO(I)} \leq \eps,\; \av{\phi}{I}=x_1,\; \av{\phi^2}{I}=x_2}.\label{OldBellman-}
\end{align}
The latter two functions were studied in detail in~\cite{SV}. We survey these results since they will play an important role in our study. 

\subsection{Description of~$\Bell_{p;\eps}^\pm$}

The domain of both functions~$\Bell^\pm_{p;\eps}$ is
\begin{equation*}
\Omega^2_\eps = \Set{(x_1,x_2) \in \mathbb{R}^2}{x_1^2 \leq x_2 \leq x_1^2+\eps^2}.
\end{equation*}
By the domain of a Bellman function we mean the set of~$x$ where the function is not equal to~$-\infty$. In other words, the set of functions~$\varphi$ over which we optimize in formulas~\eqref{OldBellman} and~\eqref{OldBellman-} is non-empty for these~$x$ (there exists at least one~$\varphi$ such that~$\av{\varphi}{I} = x_1$,~$\av{\varphi}{I} = x_2$, and~$\|\varphi\|_{\BMO(I)} \leq \eps$). Both functions also satisfy the boundary condition~$\Bell^\pm_{p;\eps}(t,t^2) = |t|^p$, $t \in \mathbb{R}$. % on the lower part of the boundary of~$\Omega^2_\eps$. 
From now on we omit the index~$\eps$ in the notation of domains and functions. %We also restrict our attention to the case~$2\leq p \leq r$.

To describe~$\Bell^\pm$, we need some auxillary functions.
For $p>1$ and $u \geq 0$ define
\eqlb{mp+}{
m_p(u) = \frac{p}{\eps}\int_{u}^{+\infty}\!\!\! e^{(u-t)/\eps} t^{p-1} dt,
}
\eqlb{mp-}{
k_p(u) = \ \frac{p}{\eps}\int_\eps^u e^{(t-u)/\eps} t^{p-1} dt.
}

For any $u \in \mathbb{R}$ we denote the segment connecting the points~$(u,u^2)$ with~$(u+\eps, (u+\eps)^2+\eps^2)$ by~$S_+(u)$ and the segment connecting~$(u,u^2)$ with~$(u-\eps, (u-\eps)^2+\eps^2)$ by~$S_-(u)$. Note that these segments touch upon the upper boundary of~$\Omega^2$, that is the parabola~$x_2=x_1^2+\eps^2$. For any $(x_1,x_2) \in \Omega^2$ there exist unique $u_\pm = u_\pm(x_1,x_2) \in \mathbb{R}$ such that $(x_1,x_2) \in S_{\pm}(u_\pm)$, $u_+\leq u_-$.

\newcommand{\Am}[1]{A_{m_#1}}
\newcommand{\Ak}[1]{A_{k_#1}}
Define the function $\Am{p}$ on $\Omega^2$ in the following way. We put 
\eqlb{eqAm}{
\begin{aligned}
\Am{p}(x) =& u^p + m_p(u) (x_1-u), \qquad &x \in S_+(u), \qquad u\geq 0,\\
\Am{p}(x) =& |u|^p - m_p(|u|) (x_1-u),\qquad &x \in S_-(u), \qquad u \leq 0.
\end{aligned}
} 
In the triangle between the tangents~$S_-(0)$ and~$S_+(0)$, we set
\eqlb{eqAmang}{
\Am{p}(x) = \frac{m_p(0)}{2\eps} x_2, \qquad |x_1|\leq \eps,\quad 2\eps |x_1| \leq x_2 \leq x_1^2+\eps^2.
} 
Formulas~\eqref{eqAm} and~\eqref{eqAmang} define the function $\Am{p}$ on the entire domain $\Omega^2$. Note that $\Am{p}$ is $C^1$-smooth and even with respect to~$x_1$.

Define the function $\Ak{p}$ on $\Omega^2$ as follows. We put 
\eqlb{eqAk}{
\begin{aligned}
\Ak{p}(x) =& u^p + k_p(u) (x_1-u), \qquad &x \in S_-(u), \qquad u &\geq \eps,\\
\Ak{p}(x) =& |u|^p - k_p(|u|) (x_1-u),\qquad &x \in S_+(u), \qquad u &\leq -\eps.
\end{aligned}
} 
In the domain~$x_2 \leq \eps^2$, we set
\eqlb{eqAkcup}{
\Ak{p}(x_1,x_2) = x_2^{p/2}, \qquad x_1^2\leq x_2 \leq \eps^2.
}
Formulas~\eqref{eqAk} and~\eqref{eqAkcup} define the function $\Ak{p}$ on the entire domain $\Omega^2$. This function is also $C^1$-smooth and even with respect to~$x_1$.

Now we are ready to describe the functions $\Bell^\pm$:
\eqlb{}{
\Bell^+_p = 
\begin{cases}
\Am{p}, & \text{if}\quad  2 \leq p< \infty\\
\Ak{p}, & \text{if}\quad 1 < p \leq 2
\end{cases}
\qquad
\text{and}
\qquad
\Bell^-_p = 
\begin{cases}
\Ak{p}, & \text{if}\quad  2 \leq p< \infty\\
\Am{p}, & \text{if}\quad 1 < p \leq 2.
\end{cases}
}

Here we collect some useful relations for derivatives of the functions $m_p$ and~$k_p$:
\eqlb{mpp+}{
m_p''(u) = \frac{p(p-1)(p-2)}{\eps}\int_{u}^{+\infty}\!\!\! e^{(u-t)/\eps} t^{p-3} dt,
}
\eqlb{mpp-}{
k_p''(u) = \ p(p-2)\eps^{p-3}e^{(\eps-u)/\eps} + \frac{p(p-1)(p-2)}{\eps}\int_{\eps}^{u}e^{(t-u)/\eps} t^{p-3} dt,
}
\eqlb{m-diff}{
-\eps m_p'(u)+ m_p(u)= pu^{p-1}, \qquad \eps k_p'(u)+ k_p(u)= pu^{p-1},
}
\eqlb{mpdiffnew}{
\eps \Big(m_p^{(\ell+1)} + k_p^{(\ell+1)}\Big) = m_p^{(\ell)} - k_p^{(\ell)},\qquad \ell \geq 0,\\
}
where the notation $g^{(k)}$ means the $k$-th derivative of $g$.

\subsection{Simple properties of the optimization problem}
The domain of the function~$\Bell_{p,r;\eps}$ introduced in~\eqref{NewBellman} is described in terms of the functions $\Bell^\pm_{p;\eps}$ from~\eqref{OldBellman}.
\begin{St}\label{Prop121201}
The set
\begin{equation*}
\Omega_\eps^3 = \Set{x \in \mathbb{R}^3}{(x_1,x_2)\in\Omega^2_\eps,\; x_3 \in \big[\Bell^-_{p;\eps}(x_1,x_2),\Bell^+_{p;\eps}(x_1,x_2)\big]},
\end{equation*}
is the domain of the Bellman function~$\Bell_{p,r;\eps}$.
\end{St}
At this point we note that for $x \notin \Omega_\eps^3$ there are no test functions and we definitely have $\Bell_{p,r}(x) = -\infty$ in this case. On the other hand, a test function $\vf$ for any $x \in \Omega_\eps^3$ with $x_3 = \Bell^\pm_{p;\eps}(x_1,x_2)$ was constructed in~\cite{SV}. In other words, we have proved that the points on the boundary of~$\Omega_\eps^3$ belong to the domain of~$\Bell_{p,r}$. We will complete the proof of Proposition~\ref{Prop121201} after formulating Proposition~\ref{BasicProperties}. 

%\begin{proof}
%{\bf [do we need this lemma?]}
%\end{proof}
A function~$G\colon \omega \to \mathbb{R}\cup\{\pm \infty\}$, where~$\omega \subset \mathbb{R}^d$ is an arbitrary set, is called locally concave if for any segment~$\ell \subset \omega$, the restricted function~$G|_\ell$ is concave. 

We collect standard facts concerning Bellman functions of such kind.
\begin{St}\label{BasicProperties}
\begin{enumerate}
\item The function~$\Bell_{p,r}$ satisfies the boundary conditions on the skeleton of~$\Omega_\eps^3$\textup:
\eqlb{eqBCsceleton}{
\Bell_{p,r}(t,t^2,|t|^p) = |t|^r, \qquad t \in \mathbb{R}.
}
\item The function~$\Bell_{p,r}$ is locally concave on~$\Omega_{\eps}^3$.
\item The function~$\Bell_{p,r}$ is the pointwise minimal among all locally concave on~$\Omega_{\eps}^3$ functions~$G$ that satisfy the boundary condition~\eqref{eqBCsceleton}.
\end{enumerate}
\end{St}
%\begin{proof}[]
The first statement of Proposition~\ref{BasicProperties} follows from the fact that if $\av{\vf^2}{} = \av{\vf}{}^2$, then $\vf$ is constant function and $\av{|\vf|^r}{} = |\av{\vf}{}|^r$. 

The second one is not so trivial. It is a consequence of the following fact (see Corollary 3.13 in~\cite{SZ}): if $\vf_i \in \BMO(I)$ with $\|\vf_i\|_{\BMO(I)}\leq \eps$, $i=1,2$, and the segment with the endpoints $(\av{\vf_i}{I},\av{\vf_i^2}{I})$ lies in $\Omega_\eps^2$, then for any $\theta \in (0,1)$ there exists a function $\vf \in \BMO(I)$ such that 
\begin{equation*}
\big|\{s \in I\colon \vf(s)>\lambda \}\big| = \theta \cdot \big|\{s \in I\colon \vf_1(s)>\lambda \}\big| +(1-\theta)\cdot \big|\{s \in I\colon \vf_2(s)>\lambda \}\big|, \ \lambda \in \mathbb{R},\quad \hbox{and}\  \|\vf\|_{\BMO(I)}\leq \eps.
\end{equation*}

This fact also leads to the existence of a test function for any $x \in \Omega_\eps^3$. Indeed, any~$x\in \Omega_\eps^3$ may be represented as a convex combination~$x=\theta y + (1-\theta)z$, where~$y$ and~$z$ lie on the boundary of~$\Omega_\eps^3$ and~$x,y$ and~$z$ have one and the same first two coordinates. By the results of~\cite{SV}, we know that there exist test functions~$\varphi_1$ and~$\varphi_2$ for~$y$ and~$z$ correspondingly. Application of Corollary 3.13 from~\cite{SZ} cited in the previous paragraph to~$\varphi_1$ and~$\varphi_2$ produces a test function for~$x$ and completes the proof of Proposition~\ref{Prop121201}.

The third statement of Proposition~\ref{BasicProperties} is usual for the Bellman function technique and is proved by the so-called Bellman induction, see e.g.~\cite{SZ}.

%\end{proof}

In view of Proposition~\ref{BasicProperties}, it suffices to construct a $C^1$-smooth function~$G \colon\Omega^3_{\eps}\to\mathbb{R}$ such that
\begin{enumerate}[1)]
	\item the function $G$ is locally concave on~$\Omega^3_\eps$;
	\item the function $G$ fulfills the boundary conditions~\eqref{eqBCsceleton};
	\item for any point~$x \in \Omega^3_\eps$ with $x_3 = \Bell^\pm_{p;\eps}(x_1,x_2)$, there is a function~$\phi_x \in \BMO(I)$ such that 
\eqlb{eqphi_x}{
\|\phi_x\|_{\BMO(I)} \leq \eps,\; \av{\phi_x}{I}=x_1,\; \av{\phi_x^2}{I}=x_2,\; \av{|\phi_x|^p}{I}=x_3,\; \av{|\phi_x|^r}{I}=G(x);
}	
	\item for any point~$x \in \Omega^3_\eps$, there exists a two-dimensional plane $L[x] \subset \mathbb{R}^3$, $x \in L[x]$, such that $G$ is linear on the connected component of $L[x] \cap \Omega^3_\eps$ containing $x$.
\end{enumerate}
If all of the above requirements hold, then,~$G =\Bell_{p,r}$. Indeed, the inequality $G(x)\geq \Bell_{p,r}(x), x \in \Omega^3_\eps,$ follows from conditions 1), 2) and the third statement of Proposition~\ref{BasicProperties}. The reverse inequality  $G(x)\leq \Bell_{p,r}(x)$ for $x \in \Omega^3_\eps$ with $x_3 = \Bell^\pm_{p;\eps}(x_1,x_2)$ follows from condition 3) and the definition of the Bellman function~$\Bell_{p,r}$. For other $x \in \Omega^3_\eps$, the inequality~$G(x) \leq \Bell_{p,r}(x)$ is implied by condition 4): $G$ is linear on $L[x]$ while $\Bell_{p,r}$ is concave there. We will provide more details in Subsection~\ref{Sub34}.

A function~$\varphi_x$ satisfying~\eqref{eqphi_x} is called an optimizer for~$G$ at~$x$.

\section{Solution to the optimization problem}\label{SectG}
Our aim is to construct the function~$G$ on~$\Omega_\eps^3$ described at the end of the previous section. %Òàê êàê ìû âåðèì â ðàâåíñòâî~\eqref{eqBC}, ìû áóäåì ñòðîèòü ôóíêöèþ $G$, óäîâëåòâîðÿþùóþ ýòîìó ãðàíè÷íîìó óñëîâèþ. 
We split~$\Omega_\eps^3$ into three subdomains $\Xi_+, \Xi_0,\Xi_-$: 
\eqlb{defXi}{
\begin{split}
\Xi_0 &= \Set{x \in \Omega_\eps^3}{|x_1|\leq 2\eps,\; x_2 \geq 4\eps|x_1|-3\eps^2\!,\; (p-2)\big(x_3 - \eps^p - \frac{x_2-\eps^2}{4\eps}m_p(\eps)\big)\geq 0},\\
\Xi_+ &= \{x \in \Omega_\eps^3 \setminus \Xi_0 \colon x_1>0\},\\
\Xi_- &= \{x \in \Omega_\eps^3 \setminus \Xi_0 \colon x_1<0\}.
\end{split}
}
The latter condition defining $\Xi_0$ may look strange. However, it is needed to distinguish the cases $p\geq 2$ and $p<2$. %there should be  $x_3 \geq \eps^p + \frac{x_2-\eps^2}{4\eps}m_p(\eps)$ while for $p<2$ it should be reverse. 
  
We will construct~$G$ on each of these domain by an individual formula, verify the local concavity on each of the domains, and also prove that the three parts provide a~$C^1$-smooth function on the union of the domains. This will lead to a locally concave $C^1$-smooth function~$G$ on~$\Omega_\eps^3$.

\subsection{Construction on~$\Xi_+$}
Let $u \geq \eps$. %Èç òî÷êè $U = (u,u^2,u^p)$, ëåæàùåé íà îñòîâå, âûõîäÿò äâà îòðåçêà, ëåæàùèå íà âåðõíåé è íèæíåé ãðàíèöå îáëàñòè $\Omega_\eps^3$, íà êàæäîì èç êîòîðûõ ôóíêöèÿ $G$ äîëæíà áûòü ëèíåéíà. 
Consider the two-dimensional plane $L_u$ that passes through~$U = (u,u^2,u^p)$ and the points
\eqlb{Upm}{
\begin{split} 
U_+ =& \Big(u+\eps,(u+\eps)^2+\eps^2, \Am{p}\big(u+\eps,(u+\eps)^2+\eps^2\big)\Big) = \Big(u+\eps,(u+\eps)^2+\eps^2, u^p+\eps m_p(u)\Big),\\
U_- =& \Big(u-\eps,(u-\eps)^2+\eps^2, \Ak{p}\big(u-\eps,(u-\eps)^2+\eps^2\big)\Big) =\Big(u-\eps,(u-\eps)^2+\eps^2, u^p-\eps k_p(u)\Big).
\end{split}
}
The equation of~$L_u$ is
\begin{equation}\label{planeeq1}
x_3 = u^p + \frac{m_p-k_p}{4\eps}\cdot\big(x_2-x_1^2+(x_1-u)^2\big)+\frac{m_p + k_p}{2}\cdot(x_1-u).
\end{equation}
Here and in what follows, we omit the argument of~$m_p$ and~$k_p$ if this does not lead to ambiguity.

Let $T_u$ be the intersection of $\Omega^3_\eps$ with the triangle with the vertices~$U,U_-, U_+$. So, $T_u$ is a curvilinear triangle. 
We define the function~$G$  on $T_u$ by linearity:
\begin{equation}\label{GDef1}
G(x_1,x_2,x_3) = u^r + \frac{m_r-k_r}{4\eps}\cdot\big(x_2-x_1^2+(x_1-u)^2\big)+\frac{m_r + k_r}{2}\cdot(x_1-u).
\end{equation}

Note that equations~\eqref{planeeq1} and~\eqref{GDef1} completely define~$G$ on~$\Xi_+$ since the latter domain is foliated by the triangles $T_u$, when $u$ runs through~$(\eps,+\infty)$. Let us prove this. The triangle $T_\eps$ is simply the boundary between $\Xi_0$ and $\Xi_+$. Recall that for any $x \in \Xi_+$ there exist unique $u_\pm \in \mathbb{R}$ such that $(x_1,x_2) \in S_{\pm}(u_\pm)$. For $x_1,x_2$ fixed we will verify that $x_3$ defined by~\eqref{planeeq1} is a monotone function of $u\in [\max(\eps,u_+),u_-]$  (see~\eqref{eq021201} further). If $u_+\geq  \eps$, then $x_3$ as a function of $u$ runs from $\Am{p}(x_1,x_2)$ to $\Ak{p}(x_1,x_2)$ when $u \in [u_+,u_-]$. If $u_+ \leq \eps$, then it runs from 
$\eps^p + \frac{x_2-\eps^2}{4\eps}m_p(\eps)$ to $\Ak{p}(x_1,x_2)$ when $u \in [\eps,u_-]$. We have proved that the~$T_u$,~$u \in (\eps,+\infty)$, foliate~$\Xi_+$.
 
%{\bf [this needs clarification]}

In order to show the local concavity of~$G$, let us verify that the Hessian of~$G$ is either non-positive or non-negative on the entire domain~$\Xi_+$ (depending on $p$ and $r$). The restrictions of this function to the planes~$L_u$, which are always transversal to the~$x_3$ axis, are linear. Thus, it suffices to show that the second derivative of~$G$ with respect to~$x_3$ does not change its sign in~$\Xi_+$.

We differentiate~\eqref{planeeq1} with respect to~$x_3$ and get
\begin{equation*}
\begin{split}
1 = u_{x_3} \Bigg(&p u^{p-1} +  \frac{m_p'-k_p'}{4\eps}\cdot\big(x_2-x_1^2+(x_1-u)^2\big) - \\
&-\frac{m_p-k_p}{2\eps}\cdot(x_1-u)
+\frac{m_p' + k_p'}{2}\cdot(x_1-u)-\frac{m_p + k_p}{2} \Bigg),
\end{split}
\end{equation*}
which, with the help of~\eqref{m-diff}, may be restated as
\begin{equation}\label{ux3}
\begin{split}
1 = u_{x_3} \cdot\frac{m_p'-k_p'}{4\eps} \cdot\big(x_2-x_1^2+(x_1-u)^2 - 2\eps^2 \big)\stackrel{\hbox{\tiny \eqref{mpdiffnew}}}{=}\\
 =u_{x_3} \cdot\frac{m_p''+k_p''}{4} \cdot\big(x_2-x_1^2+(x_1-u)^2 - 2\eps^2 \big).
\end{split}
\end{equation}
Similarly, we differentiate~\eqref{GDef1} with respect to~$x_3$ and get
\begin{equation}\label{Gx3pre}
%\begin{split}
G_{x_3}(x_1,x_2,x_3) =  u_{x_3} \cdot\frac{m_r''+k_r''}{4} \cdot\big(x_2-x_1^2+(x_1-u)^2 - 2\eps^2 \big).
%\end{split}
\end{equation}
Equations~\eqref{ux3} and~\eqref{Gx3pre} imply
\begin{equation}\label{Gx3}
G_{x_3}(x_1,x_2,x_3) =  \frac{m_r''+k_r''}{m_p''+k_p''}.
\end{equation}

It follows from~\eqref{mpp+} and~\eqref{mpp-} that
\eqlb{eq031202}{
\sign\big(m_p''(u)\big) = \sign\big(k_p''(u)\big) = \sign(p-2)
} 
when $u\geq \eps$. We have $x_2-x_1^2+(x_1-u)^2 - 2\eps^2 \leq 0$ for $x \in T_u$. Thus,~\eqref{ux3} implies
\eqlb{eq021201}{
\sign(u_{x_3}) = -\sign(p-2), \qquad x \in T_u.
} 
Therefore,
\eqlb{eq031201}{
\sign\big(G_{x_3x_3}\big) = -\sign\big((p-2)G_{x_3u}\big).
}
Let us compute that latter sign, using formulas~\eqref{Gx3} and~\eqref{mpdiffnew}:
\eqlb{Gx3u}{
\begin{split}
G_{x_3u}(x) &=  \frac{\big(m_r''-k_r''\big)\big(m_p''+k_p''\big) - \big(m_p''-k_p''\big)\big(m_r''+k_r''\big)}{\eps\big(m_p''+k_p''\big)^2}=\\
&=\frac{2\big(m_r''k_p'' - m_p''k_r''\big)}{\eps\big(m_p''+k_p''\big)^2}
= \frac{2k_r''k_p'' \Big(\frac{m_r''}{k_r''} - \frac{m_p''}{k_p''}\Big)}{\eps\big(m_p''+k_p''\big)^2}.
\end{split}
}
Thus, we need to investigate the sign of the expression~$\frac{m_r''}{k_r''} - \frac{m_p''}{k_p''}$. Without loss of generality, we may assume~$\eps=1$ (we may always substitute~$\tilde u = u/\eps$). The following notation is convenient:
\begin{align*}
I_1 =& \int_u^{+\infty}e^{u-t}t^{p-3}dt, \phantom{\int_1^u} \qquad \tilde I_1 =\frac{\partial}{\partial p} I_1 = \int_u^{+\infty}e^{u-t}t^{p-3}\log t \,dt\\
I_2 =& \int_1^u e^{t-u}t^{p-3}dt, \phantom{\int_u^{+\infty}}\qquad \tilde I_2 =\frac{\partial}{\partial p} I_2 = \int_1^u e^{t-u}t^{p-3}\log t \,dt.
\end{align*}
Consider the derivative
\begin{equation}\label{eq2}
\begin{split}
\frac{\partial}{\partial p} &\frac{m_p''}{k_p''} = \frac{\partial}{\partial p}\left(\frac{(p-1)I_1}{e^{1-u}+(p-1)I_2}\right)= \\
&=\frac{\big(I_1+(p-1)\tilde I_1 \big)\big(e^{1-u}+(p-1)I_2\big)- (p-1)I_1 \big(I_2+(p-1)\tilde I_2\big)}{(e^{1-u}+(p-1)I_2)^2}=\\
&=\frac{\big(I_1+(p-1)\tilde I_1 \big) e^{1-u}+(p-1)^2\big(\tilde I_1 I_2- \tilde I_2 I_1\big)}{(e^{1-u}+(p-1)I_2)^2}.\\
\end{split}
\end{equation}
Note that the quantities $I_1,I_2,\tilde I_1, \tilde I_2$ are non-negative. Moreover, the inequalities
\begin{equation*}
\tilde I_1 >  I_1 \log u, \qquad \tilde I_2 \leq I_2 \log u
\end{equation*}
hold true. Therefore, the expression in~\eqref{eq2} is non-negative. Consequently, 
\eqlb{eq031203}{
\sign\left(\frac{m_r''}{k_r''} - \frac{m_p''}{k_p''}\right)=\sign(r-p).
}
Finally, we investigate the sign of $G_{x_3x_3}$:
\begin{align*}
\sign\big(G_{x_3x_3}\big) &\stackrel{\hbox{\tiny \eqref{eq031201}}}{=} -\sign\big((p-2)G_{x_3u}\big)\\
&\stackrel{\hbox{\tiny \eqref{Gx3u}}}{=}
-\sign(p-2)\cdot\sign\big(k_r''\big)\cdot\sign\big(k_p''\big) \cdot\sign\Big(\frac{m_r''}{k_r''} - \frac{m_p''}{k_p''}\Big)\\
&\stackrel{\hbox{\tiny \eqref{eq031202}, \eqref{eq031203}}}{=} \sign(r-2)\cdot\sign(p-r).
\end{align*}
 
Thus, the constructed function~$G$ is locally concave on~$\Xi_+$ provided~$(r-2)(p-r)<0$ and locally convex if~$(r-2)(p-r)>0$. 

By symmetry, we define the function~$G$ on~$\Xi_-$:
\begin{equation}\label{Symmetry}
G(x_1,x_2,x_3)=G(-x_1,x_2,x_3),\quad x\in \Xi_-.
\end{equation} 
Thus, the concavity (or convexity) of this symmetrized function on~$\Xi_-$ is the same as on $\Xi_+$.

\subsection{Construction on~$\Xi_0$}
The point~$U =(u,u^2,u^p)$ lies on the skeleton of~$\Omega_\eps^3$ for any~$u \in [0,\eps]$. Let~$\tilde{L}_u$ be the two dimensional plane that passes through~$U$,~$U_+ = \big(u+\eps,(u+\eps)^2+\eps^2, u^p+\eps m_p(u)\big)$, and~$\bar U = (-u,u^2,u^p)$. Note that the segments connecting~$U$ with $U_\pm$ lie on the boundary of~$\Omega^3_{\eps}$. %It seems that the restrictions of~$\Bell_{p,r}$ to these segments might be linear. This will indeed be the case.

The plane $\tilde L_u$ is defined by the equation
\eqlb{planeeq2}{
x_3 = u^p + \frac{x_2-u^2}{2(u+\eps)}m_p(u).
}
It contains the point $\bar U_+$ that is symmetric to~$U_+$ with respect to the coordinate plane~$x_1=0$. Let $\tilde T_u$ be the intersection of $\Omega^3_\eps$ with the quadrilateral with the vertices~$U, U_+, \bar U$ and $\bar U_+$. So, $\tilde T_u$ is a curvilinear quadrilateral for any $u \in (0,\eps)$. We define the function~$G$ by linearity on $\tilde T_u$:
\begin{equation}\label{GDef2}
G(x_1,x_2,x_3) = u^r + \frac{x_2-u^2}{2(u+\eps)} m_r(u), \qquad x\in \tilde T_u.
\end{equation}

We state that the domain $\Xi_0$ is foliated by $\tilde T_u$, $u \in [0,\eps]$. Let us show this. We first note that $\tilde T_\eps$ is the common boundary of $\Xi_0$ and $\Xi_-\cup\Xi_+$.  Recall that for any $x \in \Xi_0$ there exist unique $u_\pm \in \mathbb{R}$ such that $(x_1,x_2) \in S_{\pm}(u_\pm)$. For $x_1,x_2$ fixed we will verify that $x_3$ defined by~\eqref{planeeq2} is a monotone function of $u\in [\max(0,u_+), \min(\sqrt{x_2},\eps)]$  (see~\eqref{eq091201} further). If $x_2\leq  \eps^2$, then $x_3$ as a function of $u$ runs from $\Am{p}(x_1,x_2)$ to $\Ak{p}(x_1,x_2)$ when $u \in [\max(0,u_+), \sqrt{x_2}]$. If $x_2 \geq \eps^2$, then it runs from  
$\Am{p}(x_1,x_2)$ to $\eps^p + \frac{x_2-\eps^2}{4\eps}m_p(\eps)$ when $u \in [\max(0,u_+), \eps]$. 

%We define the function~$G$ on~$\Xi_0$ {\bf [explain why we foliate $\Xi_0$]} as
%\begin{equation}\label{GDef2}
%G(x_1,x_2,x_3) = u^r + \frac{x_2-u^2}{2(u+\eps)} m_r^+(u), \qquad x\in \tilde L_u\cap \Xi_0.
%\end{equation}

Let us verify that~$G$ is either locally concave or locally convex on the entire domain~$\Xi_0$ (depending on $p$ and $r$). Similar to the previous subsection, it suffices to investigate the sign of~$G_{x_3x_3}$.

We differentiate~\eqref{planeeq2} with respect to~$x_3$ and obtain
\begin{equation}\label{ux32}
\begin{split}
1 = &u_{x_3} \left(pu^{p-1}+\frac{-2u(u+\eps)-(x_2-u^2)}{2(u+\eps)^2}\cdot m_p +\frac{x_2-u^2}{2(u+\eps)}\cdot m_p' \right)\stackrel{\hbox{\tiny \eqref{m-diff}}}{=}\\
=&u_{x_3} \left(pu^{p-1}+\frac{-2u(u+\eps)-(x_2-u^2)}{2(u+\eps)^2}\cdot\big(\eps m_p' +pu^{p-1}\big)+\frac{x_2-u^2}{2(u+\eps)}\cdot m_p' \right)=\\
%b_{x_3} \left(pb^{p-1}+\frac{-2b(b+\eps)-(x_2-b^2)}{2(b+\eps)^2}m_p^+(b)+\frac{x_2-b^2}{2\eps(b+\eps)}(m_p^+(b)-pb^{p-1})\right)=\\
%b_{x_3} \left(\frac{-2\eps(b+\eps)+(x_2-b^2)}{2\eps(b+\eps)^2}(bm_p^+(b)-p(b+\eps)b^{p-1})\right).
=&u_{x_3} \cdot\frac{-2\eps(u+\eps)+(x_2-u^2)}{2(u+\eps)^2}\cdot u\cdot\big( m_p' -p u^{p-2}\big) \stackrel{\hbox{\tiny \eqref{m-diff}}}{=}\\
=&u_{x_3} \cdot\frac{x_2-(u+\eps)^2-\eps^2}{2(u+\eps)^2}\cdot u \cdot\big(\eps m_p''  + p(p-2) u^{p-2}\big).
\end{split}
\end{equation}

Similarly,~\eqref{GDef2} leads to
\begin{equation}\label{Gx32pre}
\begin{split}
G_{x_3} = u_{x_3} \frac{-2\eps(u+\eps)+(x_2-u^2)}{2(u+\eps)^2}\cdot u\cdot \big(m_r'-r u^{r-2}\big).
\end{split}
\end{equation}

Formulas~\eqref{ux32} and~\eqref{Gx32pre} imply
\begin{equation}\label{Gx32}
\begin{split}
G_{x_3} = \frac{m_r'-r u^{r-2}}{m_p'-p u^{p-2}}.
\end{split}
\end{equation}

For $x$ in $\tilde T_u$ we have $x_2-(u+\eps)^2-\eps^2\leq 0$. By~\eqref{mpp+}, $\sign(m_p'') = \sign(p-2)$. Therefore,
formula~\eqref{ux32} implies 
\eqlb{eq091201}{
\sign(u_{x_3}) = \sign(2-p).
} 
We obtain
\eqlb{eq091202}{
\sign(G_{x_3 x_3}) = \sign(2-p)\sign(G_{x_3u}).
}
Formula~\eqref{Gx32} implies
%\begin{equation*}
\begin{gather*}
G_{x_3u}(x) =  \frac{\big(m_r''-r(r-2)u^{r-3}\big)\big(m_p'-pu^{p-2}\big) - \big(m_p''-p(p-2)u^{p-3}\big)\big(m_r'-ru^{r-2}\big)}{\big(m_p'-pu^{p-2}\big)^2}\stackrel{\hbox{\tiny \eqref{m-diff}}}{=}\\
=\frac{\big(m_r''-r(r-2)u^{r-3}\big)\big(\eps m_p'' + p(p-2) u^{p-2}\big) - \big(m_p''-p(p-2)u^{p-3}\big)\big(\eps m_r'' + r(r-2) u^{r-2}\big)}{\big(m_p'-pu^{p-2}\big)^2}=\\
=\frac{(u+\eps) \big(p(p-2)u^{p-3}m_r''-r(r-2)u^{r-3}m_p''\big)}{\big(m_p'-pu^{p-2}\big)^2}=\\
=\frac{(u+\eps) p(p-2)r(r-2)u^{p+r-5}}{\big(m_p'-pu^{p-2}\big)^2} \left(\frac{m_r''}{r(r-2)u^{r-2}}-\frac{m_p''}{p(p-2)u^{p-2}}\right).
%
%
%\frac{2\big(m_r''k_p'' - m_p''k_r''\big)}{\eps\big(m_p''+k_p''\big)^2}
%= \frac{2m_r''k_p'' \Big(\frac{m_r''}{k_r''} - \frac{m_p''}{k_p''}\Big)}{\eps\big(m_p''+k_p''\big)^2}.
\end{gather*}
%\end{equation*}

We compute the derivative of the latter expression to investigate its sign:
\begin{equation*}%\lauel{eq3}
\begin{split}
\frac{\partial}{\partial p} \left(\frac{\eps  m_p''}{ p(p-2)u^{p-2}}\right) &\stackrel{\hbox{\tiny \eqref{mpp+}}}{=} \frac{\partial}{\partial p} \left( \frac{p-1}{u^{p-2}} \int_u^{+\infty} e^{u-t}t^{p-3}dt\right) =\\
&\stackrel{\hbox{\tiny t=us}}{=} \frac{\partial}{\partial p} \left( (p-1)\int_1^{+\infty} e^{u(1-s)}s^{p-3}ds\right)=\\
&=\int_1^{+\infty} e^{u(1-s)}s^{p-3}ds+ (p-1)\int_1^{+\infty} e^{u(1-s)}s^{p-3} \log s\, ds>0.
\end{split}
\end{equation*}
%We recall that~$r>p>2$. 
Consequently,~$\sign(G_{x_3u})=\sign((r-2)(p-2)(r-p))$ and by~\eqref{eq091202}
$\sign(G_{x_3x_3})=\sign((r-2)(p-r))$.

Thus, the constructed function~$G$ is locally concave on~$\Xi_0$ provided~$(r-2)(p-r)<0$ and locally convex if~$(r-2)(p-r)>0$. 

\subsection{Concatenations}
We have defined~$G$ on three subsets of~$\Omega_\eps^3$, namely on~$\Xi_+, \Xi_0$, and~$\Xi_-$. Now we verify that the constructed function is defined on the entire domain~$\Omega_\eps^3$ and is~$C^1$-smooth. Due to symmetry, we may study the part of~$\Omega_\eps^3$ where~$x_1>0$ only.

Note that the planes~$L_u$ (see~\eqref{planeeq1}) and~$\tilde L_u$ (see~\eqref{planeeq2}) coincide when~$u=\eps$ since~$k_p(\eps)=0$ by~\eqref{mp-}. Similarly, the values of~$G$ on that common plane delivered by formulas~\eqref{GDef1} and~\eqref{GDef2} coincide since~$k_r(\eps)=0$. Thus, we have shown that~$G$ is correctly defined on~$\Omega_\eps^3$ and is continuous on this domain.

To show that~$G$ is~$C^1$ smooth, it suffices to verify that the limits of~$G_{x_3}(x)$ as~$x$ approaches~$L_\eps$ from different sides, coincide. Indeed, once we proved this, the other directional derivatives will glue continuously since~$G$ is linear on~$L_{\eps}$. By virtue of~\eqref{Gx3} and~\eqref{Gx32}, we need to prove
\begin{equation*}
\frac{m_r''(\eps)+k_r''(\eps)}{m_p''(\eps)+k_p''(\eps)} = \frac{m_r'(\eps)-r \eps^{r-2}}{m_p'(\eps)-p \eps^{p-2}}.
\end{equation*}
This may be done as follows: %~\eqref{mpdiffnew},~\eqref{mrdiffnew} ïåðåïèñûâàåòñÿ â âèäå
\begin{equation*}
\frac{m_r''(\eps)+k_r''(\eps)}{m_p''(\eps)+k_p''(\eps)}  \stackrel{\hbox{\tiny \eqref{mpdiffnew}}}{=} 
\frac{m_r'(\eps)-k_r'(\eps)}{m_p'(\eps)-k_p'(\eps)}=%\stackrel{\hbox{\tiny \eqref{m-diff}}}{=}
\frac{m_r'(\eps)-r \eps^{r-2}}{m_p'(\eps)-p \eps^{p-2}},
\end{equation*}
where in the last identity we have used that~$k_p'(\eps)=p\eps^{p-2}-\frac{1}{\eps}k_p(\eps)=p\eps^{p-2}$, which is true by~\eqref{m-diff} and~\eqref{mp-}.

To summarize, we have proved that~$G$ is~$C^1$ smooth. Therefore, its local concavity/convexity on the parts~$\Xi_0,\Xi_\pm$ of~$\Omega_\eps^3$ implies its local concavity/convexity on the entire domain~$\Omega_\eps^3$.
%Åäèíñòâåííîå, ÷òî ìû íå ïðîâåðèëè - òî, ÷òî íà ñòûêå äâóõ îáëàñòåé çàäàíèÿ îíà ñêëåèâàåòñÿ íåïðåðûâíî è $C^1$-ãëàäêî. Íåïðåðûâíîñòü ñêëåéêè âèäíà èç ñîâïàäåíèÿ ôîðìóë~\eqref{GDef1} äëÿ $u=\eps$ è~\eqref{GDef2} äëÿ $u=\eps$ (ïðè ýòîì íàäî âîñïîëüçîâàòüñÿ òåì, ÷òî $m_r^-(\eps)=0$). $C^1$-ãëàäêîñòü ëåãêî ïðîâåðèòü, âîñïîëüçîâàâøèñü ôîðìóëàìè~\eqref{Gx3},~\eqref{Gx32} è äèôôåðåíöèàëüíûìè ñîîòíîøåíèÿìè~\eqref{m+diff},~\eqref{m-diff}.

\subsection{Optimizers}\label{Sub34}

In the previous section, we have constructed a locally concave on~$\Omega_\eps^3$ function~$G$. Let us verify that it coincides with~$\Bell_{p,r}$. For that, it suffices, given arbitrary~$x\in \Omega_\eps^3$, to construct a function~$\varphi_x$ satisfying~\eqref{eqphi_x}. Such functions are usually called optimizers. We will reason in a slightly different way. Namely, we will construct the optimizers for some specific points on the boundary of~$\Omega_\eps^3$, and then, using a concavity argument show that $G = \Bell_{p,r}$ on the entire domain $\Omega_\eps^3$.

First, we construct the optimizers for the vertices of the curvilinear triangle~$T_u$,~$u \in [\eps,+\infty)$, (see~\eqref{Upm}). The constant function $\vf \equiv u$ is an optimizer for the point $U = (u,u^2,|u|^p)$. The functions
%By the construction of~$G$, we may choose the optimizer for~$\Bell_p^+$ at the point~$(u+\eps, (u+\eps)^2+\eps^2)$ (see~\cite{SV}) as an optimizer~$\phi_{U_+}$ for~$x=U_+$, that is 
\eqlb{eq101201}{
\begin{gathered}
\phi_{U_+}(t) = -\eps\ln t + u, \qquad t \in I = (0,1];\\
%}
%We may also choose the optimizer for~$\Bell_p^-$ at the point~$(u-\eps, (u-\eps)^2+\eps^2)$ (see~\cite{SV}) as an optimizer~$\phi_{U_-}$ for~$x=U_-$, that is
%\begin{equation*}
\phi_{U_-}(t) = -\eps\chi_{[0,1/2)}(t) + \eps\chi_{[1/2,1)}(t) + \eps(1+\ln t)\chi_{[1,e^{\frac{u-\eps}{\eps}}]}(t), \qquad t \in I = [0,e^{\frac{u-\eps}{\eps}}]
\end{gathered}
}
are the optimizers for the points~$U_{\pm}$ (see~\cite{SV}). One may verify that $\|\phi_{U_\pm}\|_{\BMO(I)}=\eps$ and
\begin{align*}
\av{\phi_{U_\pm}}{I} = u \pm \eps,& \qquad \av{\phi_{U_\pm}^2}{I}= (u\pm \eps)^2+\eps^2,\\ 
\av{|\phi_{U_+}|^p}{I}=\Am{p}(u+\eps, (u+\eps)^2+\eps^2),&\qquad 
\av{|\phi_{U_+}|^r}{I}=\Am{r}(u+\eps, (u+\eps)^2+\eps^2)=G(U_+),\;\\
\av{|\phi_{U_-}|^p}{I}=\Ak{p}(u-\eps, (u-\eps)^2+\eps^2),&\qquad
\av{|\phi_{U_-}|^r}{I}=\Ak{r}(u-\eps, (u-\eps)^2+\eps^2)=G(U_-).
\end{align*}

The function $G$ satisfies the boundary condition on the skeleton. If $(r-2)(p-r)<0$, then the function $G$ is locally concave on $\Omega_\eps^3$, therefore, $\Bell_{p,r} \leq G$ pointwise. For $U_+$ and $U_-$ we have $G(U_\pm) = \av{|\phi_{U_\pm}|^r}{} \leq \Bell_{p,r}(U_\pm)$. Thus, $\Bell_{p,r}(U_\pm) = G(U_\pm)$. The function $G$ is linear on $T_u$, while $\Bell_{p,r}$ is locally concave on it. Therefore, $G \leq \Bell_{p,r}$ on $T_u$. So, we have proved that $\Bell_{p,r} = G$ on $T_u$ for $u \in [\eps,+\infty)$. Due to symmetry, $\Bell_{p,r} = G$ on $\Xi_-$.

In the similar way, one may verify that $\Bell_{p,r} = G$ on $\tilde T_u$ for $u \in [0,\eps]$. Indeed, for the vertices of $\tilde T_u$ we have the same optimisers, therefore $\Bell_{p,r} = G$ at all the vertices. Similar to the reasoning in the previous paragraph, $\Bell_{p,r}$ is locally concave on $\tilde T_u$, while $G$ is linear on it. Consequently, $G \leq \Bell_{p,r}$ on $\tilde T_u$, and so $\Bell_{p,r} = G$ there. So, we have proved that $\Bell_{p,r} = G$ on $\Omega_\eps^3$ entirely and have finally proved Theorem~\ref{BellmanTheorem}.

\section{Lower Bellman function}\label{SecLowBel}
One may also consider the lower Bellman function 
\begin{equation}
\Bell^{\mathrm{min}}_{p,r;\eps} (x_1,x_2,x_3) = \inf\Set{\av{|\phi|^r}{I}}{\|\phi\|_{\BMO(I)} \leq \eps,\; \av{\phi}{I}=x_1,\; \av{\phi^2}{I}=x_2,\; \av{|\phi|^p}{I}=x_3}.
\end{equation}

\begin{St}\label{BasicPropertiesLow}
\begin{enumerate}
\item The function~$\Bell^{\mathrm{min}}_{p,r}$ satisfies boundary condition~\eqref{eqBCsceleton} on the skeleton.
\item The function~$\Bell^{\mathrm{min}}_{p,r}$ is locally convex on~$\Omega_{\eps}^3$.
\item The function~$\Bell^{\mathrm{min}}_{p,r}$ is the pointwise maximal among all locally convex on~$\Omega_{\eps}^3$ functions~$G$ that satisfy the boundary condition~\eqref{eqBCsceleton}.
\end{enumerate}
\end{St}

If $(r-2)(p-r)>0$, then the function $G$ constructed in Section~\ref{SectG} is locally convex, and it can be proved by literally the same arguments that $\Bell^{\mathrm{min}}_{p,r} = G$ in this case.
\begin{Th}
If $(r-2)(p-r)>0,$ then $\Bell^{\mathrm{min}}_{p,r} = G,$ where the function $G$ is given by formulas~\eqref{GDef1} and~\eqref{GDef2}. 
\end{Th}

\section{Extracting the constant}\label{SETC}
We are going to compute the best possible constant~$c_{p,r}$ in the inequality
\begin{equation}\label{cpr}
\|\phi\|_{L^r(I)} \leq c_{p,r}\|\phi\|_{L^p(I)}^{\frac{p}{r}} \|\phi\|_{\BMO(I)}^{1-\frac{p}{r}},\qquad \av{\varphi}{I} = 0.
\end{equation}
Without loss of generality, we may assume~$\|\phi\|_{\BMO}=1$, so we set~$\eps=1$ throughout this section. We raise the inequality to the power~$r$:
\begin{equation*}
\int |\phi|^r \leq c_{p,r}^r \int |\phi|^p,\qquad \|\phi\|_{\BMO} = 1.
\end{equation*}

Let us search for the best possible constant~$d_{p,r}$ in the inequality
\begin{equation*}
\Bell_{p,r}(0,x_2,x_3) \leq d_{p,r}^r\cdot x_3, \quad (0,x_2,x_3) \in \Omega_1^3.
\end{equation*}
Note that~$c_{p,r}\leq d_{p,r}$ by their definitions. On the other hand, it follows from homogeneity that the constant~$d_{p,r}$ is attained at some~$\phi$ with~$\|\phi\|_{\BMO}=1$ (i.e., the optimizer for~$\Bell_{p,r}$ at~$(0,x_2,x_3)$ indeed has~$\BMO$-norm equal to one). Therefore,~$c_{p,r} = d_{p,r}$. 

Since~$x_1=0$, we will be investigating the values of~$\Bell_{p,r}$ on~$\Xi_0$. The~$x_3$ coordinate is then given by~\eqref{planeeq2}, whereas the value of~$\Bell_{p,r}$ is provided by~\eqref{GDef2}. We need to maximize
\begin{equation*}
\frac{\Bell_{p,r}(0,x_2,x_3)}{x_3} =  \frac{2(u+1)u^r+(x_2-u^2)m_r(u)}{2(u+1)u^p+(x_2-u^2)m_p(u)}, \qquad x_2 \in [u^2,1], \quad u \in[0,1].
\end{equation*}
For any~$u \in(0,1]$ fixed, the latter expression is a fraction of two linear functions of~$x_2$ whose denominator does not vanish on~$[u^2,1]$. Thus, the function in question attains its maximum at one of the endpoints. At the endpoint~$x_2=u^2$, the value is $u^{r-p}$, which does not exceed one since $r>p$. At~$x_2=1$, we obtain
\begin{equation*}
g(u) = \frac{2u^r+(1-u)m_r(u)}{2u^p+(1-u)m_p(u)}.
\end{equation*}
We claim that~$g$ is decreasing on~$[0,1]$. The sign of~$g'(u)$ coincides with the sign of
%\eqlb{eq151001}{
\begin{gather}
\notag \big(2ru^{r-1} - m_r +(1-u)m_r'\big)\big(2u^p+(1-u)m_p \big) -\big(2pu^{p-1} - m_p +(1-u)m_p'\big)\big(2u^r+(1-u)m_r \big)\\
\notag\stackrel{\hbox{\tiny \eqref{m-diff}}}{=}\big(r(1+u)u^{r-1} - u m_r \big)\big(2u^p+(1-u)m_p \big)-
\big(p(1+u)u^{p-1} - u m_p\big)\big(2u^r+(1-u)m_r \big)\\
\notag =u\Big(m_ru^{p-2}(p(u^2-1)-2u^2) - m_pu^{r-2}(r(u^2-1)-2u^2)+2(r-p)(1+u)u^{p+r-2}\Big)\\
\label{eq151001} 
=u^{p+r-1}\left(p\big(r(u^2-1)-2u^2\big) \Big(\frac{m_r}{ru^r}-\frac{m_p}{pu^p}\Big) -
2(r-p)u^2\Big(\frac{m_r}{ru^r}-\frac{1+u}{u^2}\Big)\right).
\end{gather}
%}
Note that
\eqlb{eq151002}{
\frac{m_r(u)}{r u^{r}} \stackrel{\hbox{\tiny \eqref{mp+}}}{=} \int_1^\infty e^{u(1-t)}t^{r-1}\ dt.
}
Therefore, the conditions~$r>p$ and $r>2$ imply
\eqlb{eq151003}{
\frac{m_r(u)}{r u^{r}} > \frac{m_p(u)}{p u^{p}}, \qquad \frac{m_r(u)}{r u^{r}}  > \frac{m_2(u)}{2 u^{2}} = \int_1^\infty e^{u(1-t)}t \ dt = \frac{1+u}{u^2}.
}

%The coefficient of~$m_ru^{p-2}$ in~\eqref{eq151001} is negative. So, if we estimate~$m_r$ from below with the help of~\eqref{eq151003}, we obtain
%\begin{gather}
%u^{-1}R(u)< m_p(u)u^{r-2}\frac{r}{p}(p(u^2-1)-2u^2) - m_p(u)u^{r-2}(r(u^2-1)-2u^2)+2(r-p)(1+u)u^{p+r-2}=\\
%=-2m_p(u)u^{r}\frac{r-p}{p}+2(r-p)(1+u)u^{p+r-2} = -2(r-p) u^{r+p}\Big(\frac{m_p(u)}{p u^{p}} - \frac{u+1}{u^2} \Big) \stackrel{\hbox{\tiny \eqref{eq151003}}}{<}0.  
%\end{gather}

So, the expression in~\eqref{eq151001} is negative, proving our claim that~$g$ is decreasing. This implies that~$g$ attains its maximum at~$0$, which is
\begin{equation*}
g(0) = \frac{r\Gamma(r)}{p\Gamma(p)} = \frac{\Gamma(r+1)}{\Gamma(p+1)}.
\end{equation*}
Returning to 3D coordinates, we see that the extremal value is attained at the point~$(0,1,{\frac{\Gamma(p+1)}{2}}) \in \Omega_1^3$:
\begin{equation*}
\Bell_{p,r}\Big(0,1,{\frac{\Gamma(p+1)}{2}}\Big) ={\frac{\Gamma(r+1)}{2}}. 
\end{equation*}
%Any function~$\phi$ such that~$\av{\phi}{I}=0$, $\av{\phi^2}{I}=1$, clearly, has~$\BMO$-norm at least one. Consequently, the optimizer for the extremal point has norm one and we can find a function~$\phi$ (or a sequence of functions) that turns~\eqref{cpr} (or almost turn) into equality with~$c_{p,r} = \Big(\frac{\Gamma(r+1)}{\Gamma(p+1)}\Big)^{1/r}$.

{Let us provide an optimiser $\phi_0$ at the point $(0,1,{\frac{\Gamma(p+1)}{2}})$ for the function $\Bell_{p,r}$:
%We first take an optimiser $\phi_0$ for the function $\Bell_{p,r;1}$ at the point $(0,1, \Gamma(p+1)/2)$, for example, one can take 
$$
\phi_0 (t) = 
\begin{cases}
-\ln(2-t), & t \in[1,2),\\
0, & t \in [-1,1],\\
\ln(t+2), & t \in (-2,-1]
\end{cases}
$$
on the interval $I = (-2,2)$. It satisfies the following relations:
$$
\av{\phi_0}{I}=0, \quad \av{\phi_0^2}{I}=1, \quad \av{|\phi_0|^p}{I}= \frac{\Gamma(p+1)}{2},\quad \av{|\phi_0|^r}{I}= \frac{\Gamma(r+1)}{2},\quad \|\phi_0\|_{\BMO(I)}=1. 
$$

So,~\eqref{cpr} turns into equality with~$c_{p,r} = \Big(\frac{\Gamma(r+1)}{\Gamma(p+1)}\Big)^{1/r}$ for this function $\phi_0$.
}

\section{Transference}\label{s6}
In this section we will prove Theorems~\ref{Ctheorem} and~\ref{Rtheorem}. Inequality~\eqref{MultCircle} is a direct consequence of~\eqref{eq141202} since the circle~$\BMO$-norm dominates the interval~$\BMO$-norm of the same function.
To prove~\eqref{eq141201}, we will use~\eqref{eq141202} and some standard limiting arguments in Subsection~\ref{s61}. The main difficulty here is to prove the sharpness of~\eqref{MultCircle} and~\eqref{eq141201}, we will do this in Subsection~\ref{s62}.

\subsection{Inequality}\label{s61}

We start with proving~\eqref{eq141201}. Let $I_n =[-n,n]$, $n \in \mathbb{N}$, and let~$\vf \in \BMO(\mathbb{R})$. We apply~\eqref{eq141202} to $\vf - \av{\vf}{I_n}$ on $I_n$ and use the obvious inequality $\|\vf - \av{\vf}{I_n}\|_{\BMO(I_n)} = \|\vf\|_{\BMO(I_n)} \leq \|\vf\|_{\BMO(\mathbb{R})}$:
\eqlb{eq141203}{
\|\vf - \av{\vf}{I_n}\|_{L^r(I_n)} \leq \Big(\frac{\Gamma(r+1)}{\Gamma(p+1)}\Big)^\frac{1}{r}\|\vf - \av{\vf}{I_n}\|_{L^p(I_n)}^{\frac{p}{r}}\|\vf\|_{\BMO(\mathbb{R})}^{1-\frac{p}{r}}.
} 
%It is obvious that $\|\vf - \av{\vf}{I_n}\|_{\BMO(I_n)} = \|\vf\|_{\BMO(I_n)} \leq \|\vf\|_{\BMO(\mathbb{R})}$. 
\begin{Le}\label{Lem161201}
For $\vf \in L^p(\mathbb{R})$ one has 
\eqlb{eq141204}{
\|\av{\vf}{I_n}\chi_{I_n}\|_{L^p(\mathbb{R})} = |I_n|^{1/p}\;|\av{\vf}{I_n}| \to 0, \qquad n \to \infty.
}
\end{Le}
\begin{proof} 
Let $\delta>0$ be an arbitrary real. Pick $N$ such that $\int_{\mathbb{R} \setminus I_N}|\vf|^p < \delta^p$. Then, for any~$n>N$, we have
\begin{gather*}
|\av{\vf}{I_n}| \leq \frac{|I_N|}{|I_n|}\av{|\vf|}{I_N} + \frac{|I_n|-|I_N|}{|I_n|}\av{|\vf|}{I_n\setminus I_N} \leq 
\frac{|I_N|}{|I_n|}\av{|\vf|^p}{I_N}^{1/p} + \frac{|I_n|-|I_N|}{|I_n|}\av{|\vf|^p}{I_n\setminus I_N}^{1/p}\\
\leq \frac{|I_N|^{1-1/p}}{|I_n|} \|\vf\|_{L^p(\mathbb{R})} + \frac{(|I_n|-|I_N|)^{1-1/p}}{|I_n|}\;\delta.
\end{gather*}
This proves~\eqref{eq141204}.
\end{proof}
%Lemma~\ref{Lem161201} and the triangle inequality imply the following.
\begin{Cor}
For $\vf \in L^p(\mathbb{R})$ one has 
\eqlb{eq161201}{
\|\vf - \av{\vf}{I_n}\|_{L^p(I_n)} \to \|\vf\|_{L^p(\mathbb{R})}, \qquad n \to \infty.
}
\end{Cor}
\begin{proof}
From~\eqref{eq141204}, we have
$$
\lim_{n \to \infty} \|\vf - \av{\vf}{I_n}\|_{L^p(I_n)} = 
\lim_{n \to \infty} \|\vf\chi_{I_n} - \av{\vf}{I_n}\chi_{I_n}\|_{L^p(\mathbb{R})}=
\lim_{n \to \infty} \|\vf\chi_{I_n}\|_{L^p(\mathbb{R})} = 
\|\vf\|_{L^p(\mathbb{R})}.
$$
\end{proof}
We finish the proof of~\eqref{eq141201} by using Fatou's Lemma:% For any $m$ fixed we obtain
%\eqlb{eq141205}{
\begin{align*}
\|\vf\|_{L^r} &\leq \liminf_{n \to \infty} \|\vf - \av{\vf}{I_n}\|_{L^r(I_n)}\\
&\stackrel{\hbox{\tiny \eqref{eq141203}}}{\leq}
\lim_{n\to \infty} \Big(\frac{\Gamma(r+1)}{\Gamma(p+1)}\Big)^\frac{1}{r}\|\vf - \av{\vf}{I_n}\|_{L^p(I_n)}^{\frac{p}{r}}\|\vf\|_{\BMO(\mathbb{R})}^{1-\frac{p}{r}} 
= \Big(\frac{\Gamma(r+1)}{\Gamma(p+1)}\Big)^\frac{1}{r}\|\vf\|_{L^p(\mathbb{R})}^{\frac{p}{r}}\|\vf\|_{\BMO(\mathbb{R})}^{1-\frac{p}{r}}.
\end{align*}
%}
%Tending $m$ to $\infty$ in~\eqref{eq141205} we obtain~\eqref{eq141201}.

\subsection{Sharpness}\label{s62}
%Now we turn to proving sharpness of~\eqref{eq141201}. We first take an optimiser $\phi_0$ for the function $\Bell_{p,r;1}$ at the point $(0,1, \Gamma(p+1)/2)$, for example, one can take 
%$$
%\phi_0 (t) = 
%\begin{cases}
%-\ln(2-t), & t \in[1,2),\\
%0, & t \in [-1,1],\\
%\ln(t+2), & t \in (-2,-1]
%\end{cases}
%$$
%on a segment $I = (-2,2)$. It satisfies the following relations:
%$$
%\av{\phi_0}{I}=0, \quad \av{\phi_0^2}{I}=1, \quad \av{|\phi_0|^p}{I}= \frac{\Gamma(p+1)}{2},\quad \av{|\phi_0|^r}{I}= \frac{\Gamma(r+1)}{2},\quad \|\phi_0\|_{\BMO(I)}=1. 
%$
We will first prove the sharpness of~\eqref{MultCircle} and complete the proof of Theorem~\ref{Ctheorem}. For that, we will construct a special function~$\psi_0$ on the circle. After that, we will modify this function to prove the sharpness of~\eqref{eq141201}, thus, completing the proof of Theorem~\ref{Rtheorem}.

Unfortunately, we have no simple formula for such a function $\psi_0$. We will rely upon the material of~\cite{SZ} and~\cite{SZTransference}. %That paper worked mostly with simple functions. So, we first need to approximate~$\varphi_0$ by simple functions. By a simple function we mean a function that attains a finite number of values. First, we consider the truncations of~$\varphi$, that is the functions
%\begin{equation*}
%\varphi[N] = \min(\max(\varphi_0,-N),N),\qquad N\in\mathbb{N}. 
%\end{equation*}
%Since the~$\BMO$ norm of the truncated function does not exceed the~$\BMO$ norm of~$\varphi$, the functions~$\varphi[N]$ fulfill the requirements
%\begin{equation*}
%\|\phi[N]\|_{\BMO(I)} \leq 1, \qquad
%\|\phi[N]\|^p_{L^p(I)} = \av{|\phi_0|^p}{I} + O(\delta),\qquad
%\|\phi[N]\|^r_{L^r(I)} = \av{|\phi_0|^r}{I} + O(\delta),
%\end{equation*}
%provided~$N$ is sufficiently large and~$I = [-2,2]$. We pick~$N$ sufficiently large and approximate~$\varphi[N]$ by a simple function in~$L_{\infty}(I)$, thus obtaining a simple function~$\varphi_1$ that satisfies
%\begin{equation}\label{VarphiOne}
%\|\phi_1\|_{\BMO(I)} \leq 1 + \frac{\delta}{2}, \qquad
%\|\phi_1\|^p_{L^p(I)} = \av{|\phi_0|^p}{I} + O(\delta),\qquad
%\|\phi_1\|^r_{L^r(I)} = \av{|\phi_0|^r}{I} + O(\delta),
%\end{equation}
The following lemma is pivotal in the construction. In particular, this lemma proves the sharpness of~\eqref{MultCircle}. 
\begin{Le}\label{Le141201}
For any $\delta>0$ there exists a function $\psi_0$ on $\mathbb{R}$ with the following properties\textup{:}
\begin{enumerate}[1\textup{)}]
	\item \label{prop1}$\psi_0$ is $1$-periodic\textup{,} i.\,e.\textup{,} $\psi_0(t+1)=\psi_0(t)$ for $t \in \mathbb{R};$
	%\item \label{prop2}$\psi_0(t) = 0$ for $t \in (1,2);$
	\item \label{prop3} 
	$$
        \int_0^1 |\psi_0|^p =\av{|\phi_0|^p}{I} + O(\delta), \qquad  \int_0^1 |\psi_0|^r =  \av{|\phi_0|^r}{I} + O(\delta);
        $$
	%$$
	%\Big|\big\{t \in (0,1) \colon \psi(t)>\lambda \big\}\Big| =	\frac{1}{|I|}\cdot \Big|\big\{t \in I \colon \phi_1(t)>\lambda \big\}\Big|, \qquad \lambda \in \mathbb{R};
	%$$
	\item\label{prop4} $\|\psi_0\|_{\BMO(\mathbb{R})} \leq 1+\delta;$
%	\item \label{prop5} for any $a \in (0,1)$ for $\tilde J = [0,a]$ or $\tilde J = [a,1]$ one has
%	$$
%	|\av{\psi}{\tilde J}-\av{\phi_0}{I}| \leq \delta, \qquad |\av{\psi^2}{\tilde J}-\av{\phi_0^2}{I}| \leq \delta;
%	$$
	\item \label{prop6} $\int_0^1\psi_0 = 0$ and~$\int_0^1\psi^2_0 = 1$.
\end{enumerate}
\end{Le}
\begin{proof}
The proof follows the lines of~\cite{SZTransference}. 

Fix~$\delta > 0$.  First, we will need the notion of an~$\Omega^2_{\eps}$-martingale introduced in~\cite{SZ}. We say that a discrete time martingale~$(M,S) = (\{M_n\}_n,\{S_n\}_n)$, where~$S = \{S_n\}_n$ is a discrete time filtration of finite algebras, and the~$M_n$ are~$\mathbb{R}^2$-valued random variables, is an~$\Omega^2_{\eps}$-martingale provided
\begin{enumerate}[1)]
\item $S_0$ is the trivial algebra;
\item there exists a summable random variable~$M_{\infty}$ whose values lie on the parabola~$x_2 = x_1^2$ almost surely and such that~$M_n \to M_{\infty}$ as~$n\to\infty$ in mean and almost surely;
\item for any atom~$w\in S_n$, the convex hull of the set~$\{M_{n+1}(x)\colon x\in w\}$ lies inside~$\Omega_{\eps}^2$.
\end{enumerate}
We refer the reader to~\cite{SZ} for basic properties of such type martingales. By Theorem~$3.7$ in~\cite{SZ}, there exists an~$\Omega_{1+\frac{\delta}{3}}^2$ martingale~$M$ such that
\begin{equation*}
\mathrm{P}(M_{\infty}^1 > \lambda) = \frac{1}{|I|}\cdot \Big|\big\{t \in I \colon \phi_0(t)>\lambda \big\}\Big|, \qquad \lambda \in \mathbb{R};
\end{equation*}
by~$M_{\infty}^1$ we denote the first coordinate of the random vector~$M_{\infty}$ (recall that~$\varphi_0$ is the optimizer constructed at the end of the previous section). In other words, the first coordinate of the terminate distribution of~$M$ is equimeasurable with~$\varphi_0$.

Next, by a routine stopping time argument, we may replace~$M$ with a simple~$\Omega^2_{1+\frac{\delta}{3}}$-martingale~$N$ (i.e. a martingale that stops after a finite number of steps) such that~$N_0 = M_0$ and
\begin{equation*}
\E \big|N_{\infty}^1\big|^p = \E \big|M_{\infty}^1\big|^p + O(\delta);\qquad \E \big|N_{\infty}^1\big|^r = \E \big|M_{\infty}^1\big|^r + O(\delta).
\end{equation*}

We apply Theorem~$2.3$ from~\cite{SZTransference} to~$N$ and obtain a~$1$-periodic function~$\psi_0$ on the line such that
\begin{equation*}
\|\psi_0\|_{\BMO(\mathbb{R})} \leq 1+\frac{\delta}{2}
\end{equation*} 
and~$\psi_0$ is equimeasurable with~$N^1_{\infty}$ in the sense that
\begin{equation*}
\Big|\Big\{t \in \Big[-\frac12,\frac12\Big] \colon \psi_0(t)>\lambda \Big\}\Big| = \mathrm{P}(N_{\infty}^1 > \lambda)
\end{equation*}
for any~$\lambda \in\mathbb{R}$. In particular, the function~$\psi_0$ satisfies requirements~\ref{prop3} and~\ref{prop4} of the lemma. Since
\begin{equation*}
N_0 = M_0 = (0,1),
\end{equation*}
we have~$\av{\psi_0}{[-\frac12,\frac12]} = 0$ and~$\av{\psi_0^2}{[-\frac12,\frac12]} = 1$. 
\end{proof}
So, we have proved the sharpness of~\eqref{MultCircle} and completed the proof of Theorem~\ref{Ctheorem}.

Now we are ready to prove the sharpness of~\eqref{eq141201}, which will easily follow from the lemma below.
\begin{Le}\label{Le200101}
For any $\delta>0$ there exists a function $\psi$ on $\mathbb{R}$ with the following properties\textup{:}
\begin{enumerate}[1\textup{)}]
	%\item \label{prop1}$\psi_0$ is $1$-periodic, i.\,e., $\psi_0(t+1)=\psi_0(t)$ for $t \in \mathbb{R};$
	\item \label{prop2}$\psi(t) = 0$ for $t \notin (0,1);$
	\item \label{prop32} 	
$$
\int_0^1 |\psi|^p=\int_0^1|\psi_0|^p =\av{|\phi_0|^p}{I} + O(\delta), \qquad  \int_0^1 |\psi|^r =\int_0^1|\psi_0|^r=  \av{|\phi_0|^r}{I} + O(\delta);
        $$
	%$$
	%\Big|\big\{t \in (0,1) \colon \psi(t)>\lambda \big\}\Big| =	\frac{1}{|I|}\cdot \Big|\big\{t \in I \colon \phi_1(t)>\lambda \big\}\Big|, \qquad \lambda \in \mathbb{R};
	%$$
	\item\label{prop42} $\|\psi\|_{\BMO(\mathbb{R})} \leq 1+\delta.$
	%\item \label{prop6} $\int_0^1\psi_0(t)\,dt = 0$.
\end{enumerate}
\end{Le}
\begin{proof}
To prove the lemma, we will apply the homogenization procedure from~\cite{SZTransference} to~$\psi_0$\footnote{We thank Fedor Nazarov for suggesting to use the homogenization procedure in this context.}. Let us briefly describe it. Let~$g$ be a function on the interval~$I = [i_1,i_2]$ and let~$J$ be an interval. Define the transfer~$g_{J}$ of~$g$ to~$J$ by the rule
\begin{equation*}
g_{\scriptscriptstyle{J}}(x) = g\Big((x-j_1)\frac{i_2 - i_1}{j_2 - j_1} + i_1\Big), \quad x \in J = [j_1,j_2].
\end{equation*}

Now let~$\lambda \in (0,1)$. Consider the splitting of~$[-\frac12,\frac12]$ into subintervals\textup:
\begin{equation*}
I_{k,\pm} = \Big[\pm\frac{1-\lambda^{k-1}}{2}, \pm\frac{1 - \lambda^k}{2}\Big],\quad k\in \mathbb{N}.
\end{equation*}
Let~$g$ be a function defined on~$[-\frac12,\frac12]$. We call the function~$\Gamma_{\lambda}[g]$ defined on the same interval by the formula
\begin{equation*}
\Gamma_{\lambda}[g] = g_{I_{k,\pm}} \ \hbox{on the interval}\ I_{k,\pm},\quad k\in\mathbb{N},
\end{equation*}
the~$\lambda$-homogenization of~$g$.

Note that~$\Gamma_{\lambda}[g]$ has the same distribution as~$g$. Lemma~$2.7$ in~\cite{SZTransference} says that
\begin{equation*}
\|\Gamma_{\lambda}[\psi_0]\|_{\BMO([-\frac12,\frac12])} \leq 1+\delta,
\end{equation*} 
provided~$\lambda$ is sufficiently close to one. From now on we suppose~$\lambda$ to be sufficiently close to one to fulfill this property. Set
\begin{equation*}
\psi(x) = \Gamma_{\lambda}[\psi_0]\Big(x-\frac12\Big),\quad x\in [0,1],
\end{equation*}
and extend~$\psi$ to the whole line by zero to fulfill property~\ref{prop2}. Since the distribution of~$\psi|_{[0,1]}$ coincides with the distribution of~$\psi_0|_{[0,1]}$, we have requirement~\ref{prop32} satisfied. We also note that~$\|\psi\|_{\BMO([0,1])}\leq 1+\delta$. %Property~\ref{prop4} is satisfied by Lemma $2.4$ in~\cite{SZTransference} (in fact, the function~$\psi$ constructed here is the concatenation of~$\psi_0$ and the identical zero \cre{provided in the proof of that lemma}). 

It remains to verify property~\ref{prop42}. Let us first show that for any~$\tilde J = [0,a]$ or $\tilde J = [a,1]$,~$a\in (0,1)$, one has
\begin{equation}\label{Close}
|\av{\psi}{\tilde J}-\av{\psi_0}{[0,1]}| \leq \delta, \qquad |\av{\psi^2}{\tilde J}-\av{\psi_0^2}{[0,1]}| \leq \delta.
\end{equation}
To do this, we note the distribution of~$\Gamma_{\lambda}[\psi_0]|_{\tilde{J}}$ is close to the distribution of~$\psi_0$ on~$[0,1]$ (see the proof of Lemma~$2.7$ in~\cite{SZTransference}) for details). So, we may choose~$\lambda$ to be sufficiently close to~$1$ in such a way that~\eqref{Close} holds true.

Let us verify property~\ref{prop42}. For any interval $J\subset \mathbb{R}$ we need to prove
$$V(J) \stackrel{\hbox{\tiny def }}{=}
 \av{\psi^2}{J} - \av{\psi}{J}^2 \leq (1+\delta)^2.$$
Consider several cases:
\begin{itemize}
\item if $J \cap [0,1] = \emptyset$, then $\psi=0$ on $J$ and $V(J)=0$; 
\item the case $J \subset [0,1]$ had been already considered:
$
V(J) \leq \|\psi\|_{\BMO([0,1]}^2 \leq (1+\delta)^2;
$
\item if $\tilde J = J \cap [0,1] \ne \emptyset$, and $J \not\subset [0,1]$, then we may apply~\eqref{Close} to $\tilde J$ and obtain:
$$
V(J) \leq \av{\psi^2}{J} \leq \av{\psi^2}{\tilde J} \leq 1+\delta.
$$
\end{itemize}
The lemma is proved.
\end{proof}

%\tableofcontents

Dmitriy Stolyarov

\medskip

d.m.stolyarov@spbu.ru.

\bigskip

Vasily Vasyunin

\medskip

vasyunin@pdmi.ras.ru

\bigskip

Pavel Zatitskiy

\medskip

pavelz@pdmi.ras.ru.

\medskip

St. Petersburg State University, Department of Mathematics and Computer Science, 14th line 29b, Vasilyevsky Island, St. Petersburg, Russia, 199178.

\end{document}